\documentclass[reqno,12pt]{amsart} 
\usepackage{esint,mathrsfs,amsmath,amssymb,amsfonts}
\usepackage{verbatim}

\usepackage{color}
\usepackage[hyperfootnotes=false]{hyperref}
\hypersetup{
  colorlinks,
  citecolor=blue,
  linkcolor=red,
  urlcolor=blue}

\usepackage{enumitem}

\usepackage[left=1.2in, right=1.2in, bottom=1.5in]{geometry}
 
\newtheorem{theorem}{Theorem}[section]
\newtheorem{lemma}[theorem]{Lemma}

\newtheorem{corollary}[theorem]{Corollary}
\newtheorem*{conjecture}{Conjecture}
\theoremstyle{definition}
\newtheorem*{claim}{Claim}
\newtheorem{definition}[theorem]{Definition}

\theoremstyle{remark}
\newtheorem{remark}{Remark}
\numberwithin{equation}{section}

\newcommand{\Norm}[1]{\lVert#1\rVert}
\newcommand{\norm}[1]{\lvert#1\rvert}

\newcommand{\cL}{\mathcal{L}}

\newcommand{\N}{\mathcal{N}}

\newcommand{\W}{\mathcal{W}}
\newcommand{\Q}{\mathrm{Q}}

\newcommand{\R}{\mathbb{R}}

\newcommand{\s}{\subset}
\newcommand{\p}{\partial}
\renewcommand{\log}{\text{log}}

\newcommand{\txt}[1]{\text{#1}}

\newcommand{\dist}{\text{dist}}

\usepackage{cite}

\begin{document}

\title[Unique continuation at the boundary]{UNIQUE CONTINUATION AT THE BOUNDARY \\FOR DIVERGENCE FORM ELLIPTIC EQUATIONS ON QUASICONVEX DOMAINS}

\author{Yingying Cai}
\address{Yingying Cai: Academy of Mathematics and Systems Science, Chinese Academy of Sciences, Beijing, China.}
\email{caiyingying@amss.ac.cn}

\begin{abstract}
Let $\Omega\subset \R^d$ be a  quasiconvex  Lipschitz domain and $A(x)$ be a $d\times d$ uniformly eliiptic, symmetric matrix with Lipschitz coefficients. Assume nontrivial $u$ solves $-\nabla \cdot(A(x)\nabla u)=0$ in $\Omega$, and $u$ vanishes on $\Sigma=\p\Omega\cap B$ for some ball $B$. The main contribution of this paper is to demonstrate the existence of a countable collection of open balls $(B_i)_i$ such that the restriction of $u$ to $B_i \cap \Omega$ maintains a consistent sign. Furthermore, for any compact subset $K$ of $\Sigma$, the set difference $K \setminus \bigcup_i B_i$ is shown to possess a Minkowski dimension that is strictly less than $d - 1 - \epsilon$. As a consequence, we prove Lin's conjecture in quasiconvex domains.

\end{abstract}
\maketitle
\section{INTRODUCTION}
\subsection{Motivation} 
 Unique continuation is a fundamental property for elliptic equations. A closely related problem is that of boundary unique continuation. L.Bers conjectured that for a harmonic function $u\in C^1(\overline{\R^d_+})$ in the upper half space $\R^d_+,$ if $u=\partial_d u=0$ on a subset $U\subset\partial \R^d_+$ with positive surface measure, then $u$  is necessarily identically zero  in $\R^d_+$. However, Bourgain-Wolff \cite{BW90}  provided a  counterexample in $C^1(\overline{\R^d_+})$. In light of this counterexample, a related conjecture by Lin which is still open is the following: 

\begin{conjecture}[Lin's conjecture]
    Let u be a bounded harmonic function in a Lipschitz domain $\Omega\subset \R^d$. Suppose that u vanishes on a relatively open set $U\subset\partial\Omega$ and $\nabla u$ vanishes in a subset of $U $ with positive surface measure. Then $u=0$ in $\Omega$.
\end{conjecture}
 The above full conjecture remains unresolved for $d\ge 3$ in general domains. However, significant efforts have been focused on the regularity and geometric condition of $\partial\Omega$ and deriving  quantitative estimate of the singular sets on $\partial\Omega$. In \cite{L91}, Lin proved the conjecture for $C^{1,1}$ domains by showing that the singular sets have Hausdorff dimension at most $d-2$ (for additional quantitative estimates, refer to Burq-Zuily in \cite{NZ23}). This result was further expaned to $C^{1,\alpha}$ domains by Adolfsson-Escauriaza \cite{AE97}, and to $C^{1,\txt{Dini}}$ domains by Kukavica-Nystr$\dot{o}$m \cite{KN98} and Banerjee-Garofalo \cite{BG16}. Recently, Tolsa in \cite{T21} successfully proved the aforementioned conjecture for $C^1$ domains (or Lipschitz domains with small constant).  His proof leverages the new potent methodologies developed by Logunov-Malinikova in \cite{L18b,L18a,L18}.  In more recent advancements, Gallegos \cite{G22} introduced the set $\mathcal{S}_{\Sigma}'(u) =\left \{x \in \Sigma \mid u^{-1}(0) \cap B(x,r)\cap\Omega \ne \emptyset, \text{ for all }r>0   \right \},$
 which represents the region where $u$ changes sign in every neighborhood, and demonstrated that the Hausdorff dimension of $\mathcal{S}_{\Sigma}'(u) $ is strictly smaller than $d-1$. Furthermore, in the case $\Omega$ is a $C^{1,\txt{Dini}}$ domain, $\mathcal{S}_{\Sigma}'(u)$ coincides with the usual singular set at the boundary of $u:\mathcal{S}_{\Sigma}(u)=\left \{ x\in\Sigma|\norm{\nabla u(x)}=0 \right \} $ \cite[Proposition 1.9]{G22}. Along another line, Adolfsson-Escauriaza-Kenig \cite{AEK95} proved the conjecture for convex domains. Additionally, Mccurdy \cite{M19} obtained more quantitative result in arbitrary convex domains.

 Another related issue pertains to the quantitative estimation of nodal sets near the boundary. Logunov-Malinnikava-Nadirashvili-Nazarov \cite{LMNN21} obtained the sharp upper bound for Laplace eigenfunctions with Dirichlet boundary conditions in $C^1$ domains (or Lipschitz domains with small constants). Recently, Zhu-Zhuge \cite{ZZ23} have broadened this result to quasiconvex Lipschitz domains, which encompasses both $C^1$ and convex domains. Moreover, there exists a quasiconvex Lipschitz curve that is nowhere convex or $C^1$ (see \cite[Appendix A.1]{ZZ23}). Regarding this, the concept of quasiconvex domains was initially introduced by Jia-Li-Wang in \cite{JLW10} for analyzing global Calderón-Zygmund estimate.

  It is observed that Lin's conjecture is valid for both $C^1$ and convex domains, and quantitative estimates of nodal sets have been derived for $C^1$ and quasiconvex domains. Consequently, we pose the question of whether Lin's conjecture can be demonstrated in quasiconvex domains. The main contribution of this paper is to establish a quantitative estimate of singular sets, which also leads to the verification of Lin's conjecture for quasiconvex Lipschitz domains.

\subsection{Assumptions and main results} 
Let us consider the second order elliptic operator given by $\cL=-\nabla \cdot A\nabla.$ A function $u\in H^1(\Omega)$ is termed $A$-harmonic in $\Omega$ if $\cL(u)=0$ in $\Omega$. In the context of this paper, we make the standard assumption that the coefficient matrix $A$ fulfills the following conditions:
\begin{itemize}[label=\textbf{\textbullet}]
    \item[(i)] there exists $\Lambda\ge 1$ such that $\Lambda^{-1}\mathbin{I}\le A(x)\le \Lambda\mathbin{I} $ for any $x\in\Bar{\Omega}$.

    \item[(ii)] $A^T=A$.
    \item[(iii)]there exists $\gamma\ge 0$ such that $\norm{A(x)-A(y)}\le \gamma\norm{x-y}$ for every $x,y\in \Bar{\Omega}$.
\end{itemize}

Next, we introduce the assumption on the domain $\Omega\s\R^d$. 
\begin{definition}
  We say a domain $\Omega$ is Lipschitz, if there exists $r_0\ge 0$ such that for every $x_0\in \p\Omega$ the boundary patch $\p\Omega \cap B_{r_0}(x_0)$, after a rigid transformation, can be expressed as a Lipschitz graph $x_d=\phi(x')$ such that 
  \begin{equation}\label{deflipschitz}
      B_{r_0}(x_0)\cap \Omega=B_{r_0}(x_0)\cap \left \{ x=(x',x_d):x_d\ge \phi(x') \right \} 
  \end{equation}
   Moreover, the Lipschitz constants of these graphs are uniformly bounded by some contant $L$.  
\end{definition}

In this paper, we will call $\omega(\rho):(0,\infty)\to [0,\infty)$ is a quasiconvexity modulus, which is a continuous nondecreasing function such that $\lim_{\rho \to 0}\omega(\rho)=0$.

\begin{definition}\label{def1.1}
    A Lipschitz domain $\Omega$ is called quasiconvex if it satisfies the following property: there exist a quasiconvexity modulus $\omega(\rho)$ and $r_0>0$ such that for each point $x_0$ is translated to 0 and the local graph of $\p\Omega$,  $x_d=\phi(x')$, satisfies $\phi(0)=0$ and 
    \begin{equation}\label{def qua}
        \phi(x')\ge -\norm{x'}\omega(\norm{x'}),\quad \txt{for all $x'$ with } \norm{x'}<r_0.
    \end{equation}
\end{definition}

In this paper, we adopt an alternative and equivalent definition of quasiconvex domains: : for each $x_0\in \p\Omega$ and $r<r_0$, there exists a unit vector $n$(coinciding with the outerward normal at $x_0$ if exists) such that
\begin{equation}\label{quasiconvex}
    \Omega\cap B_r(x_0)\s \left \{ y:(y-x_0)\cdot n\le r\omega(r) \right \}.
\end{equation}
Specifically, when $\Omega$ is convex, we have $\omega(r)=0$. If $\Omega$ is a $C^1$ domain, then(\ref{def qua}) should be strengthened to a two-sided condition $$\norm{x'}\omega(\norm{x'})\ge \phi(x')\ge -\norm{x'}\omega(\norm{x'}).$$  Therefore, in a manner of speaking, quasiconvexity can be viewed as a one-sided $C^1$ condition.

The main result of this paper is stated as follows.

\begin{theorem}\label{Minskithm1}
   Let $\Omega \subset \mathbb{R}^d$ be a bounded quasiconvex Lipschitz domain, and let $A$ satisfy the standard assumptions. Suppose $u$ is a nontrivial solution of $\cL(u)=0$ in $\Omega$. Then there exist positive constants $r_0=r_0(\omega,\gamma,\Lambda,L,d)$ and $\epsilon_1=\epsilon_1(\omega,\gamma,\Lambda,L,d)$ such that for all $x\in\partial\Omega$, If $u$ vanishes on $\partial\Omega\cap B(x,2r_0)$, where $\partial\Omega\cap B(x, r_0)$ is denoted as $\Sigma$. Then we can find a family of open balls $B_i$, indexed by $i\in\mathbb{N}$ and centered on $\partial\Omega\cap B(x,2r_0)$,  satisfying the following properties:
   
\begin{itemize}
    \item[(i)] For every $i\in\mathbb{N}$, $|u|>0$ in $B_i\cap \Omega$,
    \item[(ii)] For every compact $K\subset \Sigma$, the set $K\setminus  \bigcup_{i}B_i$ has Minkowski dimension at most $d-1-\epsilon_1$.
\end{itemize}
\end{theorem}

Recall that the upper Minkowski dimension of a bounded set $E\subset \R^{d-1}$ can be defined as 
\begin{equation}\label{defMinkovski}
\text{dim}_{\overline{\mathcal{M}}}E=\limsup_{j\to \infty } \frac{\log\left(\#\left\{\text{dyadic cubes $Q$ of side length $2^{-j}$ that satisfy $Q \cap E \ne \emptyset$}\right\}\right)}{j\log 2} 
\end{equation}
The proof of Theorem \ref{Minskithm1} relies on two critical technical lemmas. The first one is Lemma \ref{keylemma1}, which is a modification of lemmas found in \cite[Lemma 3.1]{T21} and \cite[Lemma 4.1]{G22}. The second key lemma, Lemma \ref{keylemma2} by Zhu-Zhuge \cite{ZZ23}, plays a crucial role in bounding the sizes of balls centered at the boundary that do not contain zeros of \(u\). Subsequently, we employ a combination technique similar to that used in \cite{G22} to regulate the size of the zero set of $u$ in close proximity to the boundary.

As a direct consequence of the previous theorem, we obtain the following corollary:
\begin{corollary}\label{corolaryofminkovski}
Under the same assumptions with Theorem \ref{Minskithm1}, we have $$ \text{dim}_{\overline{\mathcal{M}}}(\mathcal{S}_{\Sigma}'(u)\cap K)\le d-1-\epsilon_1$$
 for any compact set $K\subset\Sigma.$
\end{corollary}
\begin{remark}
    Corollary \ref{corolaryofminkovski} give Hausdorff dimension estimates for the set $\mathcal{S}_{\Sigma}'(u)$ by taking an exhaustion of $\Sigma$ by compact sets. Note that Hausdorff dimension estimates are weaker than Minkowski dimension estimates. See \cite[Chapter 5]{M95} for more about Hausdorff measure and Minkowski dimension.
\end{remark}

 Theorem \ref{Minskithm1} further enables us to establish the validity of Lin's conjecture for quasiconvex Lipschitz domains. The precise formulation of this result is as follows:
\begin{corollary}\label{quasiconvexLinconjecture}
    Under the same assumptions with Theorem \ref{Minskithm1}, the set $$\mathcal{S}_{\Sigma}(u):=\left \{x\in \Sigma\mid \norm{\nabla u(x)}=0  \right \}$$ has $d-1$-Hausdorff measure $0$. 
\end{corollary}
Note that in \cite[Section 9]{G22}, Gallegos presented an example that demonstrates the impossibility of improving Corollary \ref{quasiconvexLinconjecture} in terms of Hausdorff dimension estimates. This stands in contrast to the case of higher regularity ($C^{1,\text{Dini}}$), where the set $\mathcal{S}_{\Sigma}(u)$ is known to have dimension $(d-2)$, as shown in \cite{KZ22}. Consequently, this observation suggests that the set $\mathcal{S}_{\Sigma}'(u)$ could serve as a natural substitute for the conventional singular set $\mathcal{S}_{\Sigma}(u)$ in the context of Lipschitz domains.

\textbf{Notations.} Throughout the paper, we denote by $C,C_1, C_0, C^*,\cdots,$ large positive constants depending at most on $d, L, A,$ and $\omega$, which may change from line to line.
The notation $A\preceq B$ is equivalent to $A\le CB$, and $A\approx B$ is equivalent to $\frac{1}{C}A\le B\le CA$. Denote by $B(E,r) = \{x: d(x,E) < r \}$ the $r$-neighborhood of the set $E$.
Occasionally, we also employ the notation $A\ll B$ to indicate that $CA\le B$.

\textbf{Organizations.} The remainder of the paper is organized as follows. In Section \ref{sectionmonotonicityofdoublingindex}, we introduce tools concerning the monotonicity of the doubling index, which will be applied consistently throughout the subsequent sections. Section \ref{Whitneydecomposition} is devoted to the construction of a Whitney cuboid structure for $\Omega$, which will be utilized in the subsequent sections. In Section \ref{sectionnozeros}, we delve into the details of Lemma \ref{keylemma2}. Finally, the main theorem is proven in Section \ref{sectionoftheorem}.

\textbf{Acknowledgements.} The author is supported by National Key R\&D Program of China, Grant No.2021YFA1000800. The author is also grateful to Prof. Liqun Zhang and Prof. Jinping Zhuge for their invaluable guidance and insightful discussions. Their expertise and mentorship have been instrumental in shaping this work.

\section{ALMOST MONOTONICITY OF DOUBLING INDEX}\label{sectionmonotonicityofdoublingindex}
 An important tool to study unique continuation of elliptic PDEs is the so-called frequency function. The idea goes back to the work of Almgren \cite{a79}, where it was introduced for the case of harmonic functions. It was generalized to solutions of second order elliptic equations by Garofalo-Lin \cite{GL86}.  In this section, we will show that for quasiconvex Lipschitz domains and variable coefficients, we have the almost monotonicity for the doubling index.

\subsection{Frequency function and three-ball inequality}
Let $\Omega$ be a quasiconvex Lipschitz domain and $A$ satisfy the standard assumptions. In this subsection, we further assume (up to an affine transformation) that $0\in \Bar{\Omega}$ and $A(0)=\mathbin{I}$. Note that under this normalization, the constants for $\Omega$ and $A$ will change. In particular, the new quasiconvexity modulus $\Tilde{\omega}(r)$ will satisfy $\Tilde{\omega}(r)\le \Lambda\omega(\Lambda^\frac12 r)$ which is harmless up to a constant. Similarly $\Tilde{\gamma}\le \Lambda\gamma$ and $\Tilde{L}\le \Lambda L.$

Let $R>0$ and $u$ be an $A$-harmonic function in $B_{R,+}=B_R(0)\cap \Omega$ and satisfy the Dirichlet boundary condition $u=0$ an $B_R(0)\cap \p\Omega.$  Define 
$$\mu(x)=\frac{x\cdot A(x)x}{\norm{x}^2}\quad  \txt{and} \quad H(r)=\int_{\p B_r(0)\cap \Omega}\mu u^2d\sigma.$$

Define $$ D(r)=\int_{B_{r,+}}A\nabla u\cdot\nabla u$$.

Define the frequency function of u centered at 0 as 
\begin{equation}
    \N_u(0,r)=\frac{rD(r)}{H(r)}.
\end{equation}

Let us begin by revisiting a monotonicity lemma originally derived for the domain's interior, as in \cite{GL86}. The proof for the boundary scenario has been detailed in \cite{ZZ23}, offering a thorough examination of the boundary conditions. This lemma forms a fundamental basis for our forthcoming analysis.
\begin{lemma}\label{Montic}
    Under the above assumptions, if in addition, 
    \begin{equation}\label{starshape}
        n(x)\cdot A(x)x\ge 0,\quad \txt{for almost every}\quad x\in B_R(0)\cap\p\Omega,
    \end{equation}
    then there exists $C=C(\Lambda,d)>0$ such that $e^{C\gamma r}\N_u(0,r)$ is nondecreasing in $(0,R). $ 
\end{lemma}

We say $B_{R,+}$ is $A$-starshaped with respect to 0 if (\ref{starshape}) holds. This additional geometric assumption on $\p\Omega$ is necessary in the proof of Lemma \ref{Montic} .

The relationship between $H(r)$ and the frequency function can be easily seen from  \cite[(A.7)]{ZZ23}
\begin{equation}\label{H'}
    \norm{\frac{H'(r)}{H(r)}-\frac{d-1}{r}-\frac{2}{r}\N_u(0,r)}\le C\gamma .
\end{equation}

Moreover, by using (\ref{H'}) and Lemma \ref{Montic}, we can obtain that for any $0<r_1<r_2<r_3<R$, the following three-ball inequality holds:
\begin{equation}\label{3sphere}
    \log\frac{\int_{B_{r_2,+}}\mu u^2}{\int_{B_{r_1,+}}\mu u^2}\le \beta \log\frac{\int_{B_{r_3,+}}\mu u^2}{\int_{B_{r_2,+}}\mu u^2}+d\log\frac{r_2^{1+\beta}}{r_3^\beta r_1}+C\gamma r_3,
\end{equation}
Here, $\beta=e^{C\gamma r_3}\log\frac{\frac{r_2}{r_1}}{\frac{r_3}{r_2}}$. This inequality is commonly known as the three-ball inequality.
\subsection{Non-identity matrix}
The focus of this subsection is to eliminate the requirement of $A(0)=\mathbin{I}$ in order to define the doubling index at any given point. Similar to the works of \cite{G22} and \cite{ZZ23}, we make the assumption that $A(0)\ne \mathbin{I}$.
Since $A(0)$ is symmetric and positive definite, we can write $A(0)=\mathcal{O}D\mathcal{O}^T,$ where $\mathcal{O}$ is an (constant) orthogonal matrix and D is a (constant) diagonal matrix with entries contained in $[\Lambda^{-1},\Lambda].$ Let $E=\mathcal{O}D^\frac12\mathcal{O}^T.$ Thus $A(0)=E^2$ and $E$ is also symmetric. Let $u$ be A-harmonic in $B_R(0)\cap \Omega.$ Then we can normalize the matrix $A$ such that $A(0)= \mathbin{I}$ by a change of variable $x\to Ex.$ Precisely, let $$\Tilde{u}(x) =u(Ex)\quad \txt{and} \quad \Tilde{A}(x)=E^{-1}A(Ex)E^{-1}.$$ Then $\Tilde{u}$ is $\Tilde{A}$-harmonic in $E^{-1}(B_R\cap\Omega)=E^{-1}(B_R)\cap E^{-1}(\Omega).$ Obviously now $\Tilde{A}(0)=\mathbin{I}.$ Note that the A-starshape condition $\Tilde{n}(x)\cdot \Tilde{A}(x)x\ge 0$ on $E^{-1}(B_R\cap\p \Omega)$ is equivalent to 
\begin{equation}\label{A-star}
    n(x)\cdot A(x) A^{-1}(0)\ge 0,\quad \txt{on} \quad B_R\cap\p\Omega.
\end{equation}
This condition is consistent with (\ref{starshape}) if $A(0)= \mathbin{I}$ and invariant under linear transformations. 

Let $\Tilde{\Omega}=E^{-1}(\Omega)$. We would like to find the right form of the doubling index under a change of variables. By definition of the weight $\Tilde{\mu}$ corresponding to $\Tilde{A}$ and the change of variable $y=Ex$, 
\begin{equation}\begin{aligned}
    \int_{ B_r(0)\cap \Tilde{\Omega}}\Tilde{\mu}\Tilde{u}^2=&\int_{ B_r(0)\cap \Tilde{\Omega}}\frac{x\cdot E^{-1} A(Ex)E^{-1}x}{\norm{x}^2}(u(Ex))^2dx\\
    &=\norm{\mathrm{det} A(0)}^{\frac 12}\int_{ A^{\frac 12}(B_r(0))\cap \Omega}\frac{y\cdot A(0)^{-1}A(y)A(0)^{-1}y}{y\cdot A(0)^{-1}y}u(y)^2 dy,
\end{aligned}  
\end{equation}

Therefore, taking translations into consideration, for any $x_0\in \Bar{\Omega}$, we define 

\begin{equation}\label{defJ1}
\begin{aligned}
\mu (x_0,y) &= \frac{(y-x_0)\cdot A(x_0)^{-1}A(y)A(x_0)^{-1}(y-x_0)}{(y-x_0)\cdot A(x_0)^{-1}(y-x_0)}, \\
F(x_0,r) &= x_0+A^{\frac{1}{2}}(x_0)(B_r(0)),
\end{aligned}
\end{equation}

and 

\begin{equation}\label{defJ2}
    J_u(x_0,r)=\norm{\mathrm{det} A(x_0)}^{-\frac 12}\int_{ F(x_0,r)\cap \Omega}\mu(x_0,y)u(y)^2dy.
\end{equation}

Define the doubling index by $$ N_u(x_0,r)=\log \frac{J_u(x_0,2r)}{J_u(x_0,r)}.$$
Note that $\Lambda^d\ge \mathrm{det} A(x_0)\ge \Lambda^{-d}$ for any $x_0$.  We point out that $J_u(x_0,r)$, invariant under affine transformations, is the right form of the weighted $L^2$ norm for defining doubling index.

\subsection{Auxiliar lemmas  of doubling index}
In this subsection, we consider a Lipschitz domain $\Omega$ with $0\in \Omega$. Let $u$ be an $A$-harmonic function in $B_{1}\cap \Omega$ satisfying the boundary condition $u=0$ on $B_{1}\cap \partial\Omega$. Furthermore, we only require the domain \(\Omega\) to be Lipschitz, without the additional condition of being quasiconvex. The lemmas we present here are similar to the lemmas by Zhu and Zhuge in their work \cite{ZZ23}, but with a different focus. While Zhu and Zhuge primarily study the behavior of the doubling index near quasiconvex boundary points, our lemmas concentrate on the behavior of the doubling index at interior points.

Next, we demonstrate the property of almost monotonicity for the doubling index.
\begin{lemma}\label{Keymdoub}
    Under the above assumptions, let $ x_0\in \Omega\cap B_{\frac{1}{2}}$ and $R$ be such that $0<R\le \frac{1}{16\Lambda}$. If,  in addition $B(x_0,8\Lambda R)\cap \Omega$  is A-starshaped with respect to its center $x_0$, then there exists $C=C(\Lambda, d)$ such that for any $0\le r\le R$, the following inequality holds:
    \begin{equation}\label{keymdoub}
    N_u(x_0,r)\le (1+C\gamma r)N_u(x_0,2r)+C\gamma r.
    \end{equation}
\end{lemma}
\begin{proof}
    By the three-ball inequality in (\ref{3sphere}), if $0< r_1\le r_2\le r_3\le 8R$, we have
    $$ \log\frac{J_u(x_0,r_2)}{J_u(x_0,r_1)}\le \beta \log\frac{J_u(x_0,r_3)}{J_u(x_0,r_2)}+d\log\frac{r_2^{1+\beta}}{r_1 r_3^\beta}+C\gamma r_3.$$
    Now, we would like to obtain a monotonicity formular of the doubling index centered at $x_0.$ Let $$ r_1 = r,\quad  r_2 = 2r,\quad r_3 = 4r.$$
    It follows that $\norm{\beta -1}\le C\gamma r$ and $$ d\log\frac{r_2^{1+\beta}}{r_1 r_3^\beta}+C\gamma r_3\le C\gamma r.$$ 
    Consequently, we obtain $$ N_u(x_0,r)\le (1+C\gamma r)N_u(x_0,2r)+C\gamma r.$$
\end{proof}
In the following lemma, we provide a quantitative analysis of the propagation of the doubling index when we shift the center from one point to a nearby point. This result sheds light on the behavior of the doubling index as the center of consideration is moved within close proximity.

\begin{lemma}\label{sshiftcenter}
Let $x_1, x_0\in \Omega\cap B_{\frac{1}{2}}, 0<R\le \frac{1}{16\Lambda}$, and $B(x_0,8\Lambda R)\cap  \Omega$ is A-starshaped with respect to its center $x_0$, then there exists $C=C(\Lambda, d), C^*=C^*(\Lambda, d)$ (sufficiently large) such that if $\theta=\norm{x_1-x_0}\le \frac{R}{C^*}$, we can obtain 
    \begin{equation}\label{shiftcenter}
        N_u(x_1, R)\le (1+C(\gamma R+\frac{\theta}{R}))N_u(x_0,2R)+C(\gamma R +\frac{\theta}{R}).
    \end{equation}
\end{lemma}
\begin{proof}
    The proof follows from the Lispchitz continuity of $\mu(x_0, y)$ and $A^{\frac{1}{2}}(x_0)$ and the three-ball inequality in $x_0$. Note that $$\mu (x_0,y) = 1+ \frac{(y-x_0)\cdot A(x_0)^{-1}(A(y)-A(x_0))A(x_0)^{-1}(y-x_0)}{(y-x_0)\cdot A(x_0)^{-1}(y-x_0)}. $$ 
    Then fixing $y$, one can show that $\norm{\nabla\mu(x_0,y)}\le C\gamma.$ Thus, $$ \norm{\mu(x_0,y)-\mu(x_1,y)}\le C\gamma\theta \mu(x_1,y),$$ which yields 
    \begin{equation}\label{use1}
        \mu(x_1,y)\le (1+C\gamma \theta)\mu(x_0,y).
    \end{equation}
    
On the other hand, we see that $\norm{\nabla A^{\frac{1}{2}}}\le C\norm{\nabla A}\le C\gamma,$ where $C$ depends only on $d$ and $\Lambda$.
Thus we have
\begin{equation*}
    \begin{aligned}
        &F(x_1,2R)\subset F(x_0, 2R(1+C_0\gamma \theta)+ C_1\theta)\\
        &F(x_1,R)\supset F(x_0, R(1-C_0\gamma \theta)- C_1\theta).
    \end{aligned}
\end{equation*}

When we choose $C^*\ge 4C_1+8C_0+\Lambda$, then we obtain that
\begin{equation}\label{use2}
    F(x_1, 2R)\subset F(x_0, 2R+C^*\theta)\quad\text{and}\quad F(x_1, R)\supset F(x_0, R-C^*\theta).
\end{equation}
Moreover, $$ \norm{\txt{det} A(x_0)^{\frac{1}{2}}-\txt{det} A(x_1)^{\frac{1}{2}}}\le C\gamma \theta \le C\gamma \theta \txt{det} A(x_1)^{\frac 12}.$$
This implies 
\begin{equation}\label{use3}
  \txt{det}A(x_1)^{\frac 12}\le (1+C\gamma \theta)\txt{det}A(x_0)^{\frac 12}.
\end{equation}

It follows from (\ref{use1}), (\ref{use2}) and (\ref{use3}) that 
\begin{equation*}\begin{aligned}
    J_u(x_1, 2R)&=\norm{\txt{det} A(x_1)}^{-\frac 12}\int_{F(x_1, 2R)\cap\Omega }\mu (x_1, y)u(y)^{2}dy\\
    &\le (1+C\gamma R)^2 \norm{\txt{det} A(x_0)}^{-\frac 12}\int_{F(x_0, 2R+C^*\theta)\cap\Omega }\mu (x_0, y)u(y)^{2}dy\\
    &\le (1+C\gamma R) J_u(x_0, 2R+C^*\theta)
\end{aligned}
\end{equation*}
In the same way, we can obtain $$ J_u(x_1, R)\ge (1-C\gamma R) J_u(x_0, R-C^*\theta).$$  
Furthermore, by the three-ball inequality (\ref{3sphere}), let  $$r_1 = R-C^*\theta,\quad  r_2 = 2R+C^*\theta,\quad r_3 = 4R.$$  It follows that $\norm{\beta-1}\le C(\gamma R+\frac{\theta}{R})$ and $$ d\log\frac{r_2^{1+\beta}}{r_1 r_3^\beta}+C\gamma r_3\le C(\gamma R+\frac{\theta}{R}).$$
Hence, we obtain that 
\begin{equation*}\begin{aligned}
   N_u(x_1,R)=\log \frac{J_u(x_1, 2R)}{J_u(x_1, R)}&\le \log\frac{J_u(x_0, 2R+C^*\theta )}{J_u(x_0, R-C^*\theta)}+\log\frac{1+C\gamma R}{1-C\gamma R}\\
   &\le (1+C(\gamma R+\frac{\theta}{R}))\log\frac{J_u(x_0, 4R)}{J_u(x_0, 2R+C^*\theta)}+C(\gamma R +\frac{\theta}{R})\\
   &\le (1+C(\gamma R+\frac{\theta}{R}))N_u(x_0,2R)+C(\gamma R +\frac{\theta}{R}).
\end{aligned}    
\end{equation*}
\end{proof}

\section{DROP OF DOUBLING INDEX}\label{Whitneydecomposition}
The main result of this section is Lemma \ref{keylemma1}, which provides crucial insights into the behavior of 
doubling index near the boundary. In order to establish this result, we present a proof that bears resemblance to the approaches found in \cite{ZZ23, T21, G22, LMNN21}. 

\subsection{Whitney cuboid structure on $\Omega$}
Let $\Omega$ be a quasiconvex Lipschitz domain. Let $L\ge 0$ be the Lipschitz constant associated to the Lipschitz domain $\Omega.$ Define a standard rectangle cuboid by $Q_0=[-\frac{1}{2},\frac{1}{2})^{d-1}\times\left \{ (1+L)[-\frac{1}{2},\frac{1}{2})  \right \}.$ Throughout this paper, all the cuboids under consideration are obtained by applying affine transformations (combinations of linear transformations and translations) to $Q_0$. Let \(\pi(Q)\) represent the vertical projection of \(Q\) onto \(\mathbb{R}^{d-1}\). Consequently, \(\pi(Q)\) corresponds to a standard cube in \(\mathbb{R}^{d-1}\). For instance, \(\pi(Q_0) = [-\frac{1}{2},\frac{1}{2})^{d-1}\). Let the side lengths of \(Q_0\) and \(\pi(Q_0)\) be denoted by \(\ell(Q_0) = \ell(\pi(Q_0)) = 1\). Note that when \(L = 0\), the cuboid reduces to a cube.

Let $B$ be a ball centered in $0\in \p\Omega$ with radius $2r_0$ small enough whose boundary patch is given by $x_d=\phi(x')$ and $\norm{\nabla \phi}\le L.$  Suppose $u$ is an $A$-harmonic function in $B\cap \Omega$ with $u=0$ on $B\cap \p\Omega$, where we also let $\Sigma=\frac{1}{2}B\cap \p\Omega.$ Next, we consider the Whitney decomposition of $\Omega$ as described in \cite{T21,G22} in the coordinate framework centered at $0$.

This decomposition involves a family $\mathcal{W}$ satisfying the following properties for every $\mathrm{Q}\in \mathcal{W}$:

\begin{itemize}
    \item [(i)] $10\mathrm{Q}\subset \Omega$,
    \item[(ii)] $W\mathrm{Q}\cap \p\Omega\ne \emptyset$,
    \item[(iii)] there are at most $D_0$ cuboids $\mathrm{Q}'\in\mathcal{W}$ such that $10\Q\cap 10\Q'\ne \emptyset.$ Furthermore, for such cuboids $\Q'$, we have $\frac 12 \ell(\Q')\le \ell(\Q)\le 2\ell(\Q')$.
\end{itemize}

Above, we denote by $\ell(\Q)$ the side length of $\Q$ and by $x_Q$ the center of the cuboid $Q$. From the properties (i) and (ii) it is clear that $\dist(\Q,\p\Omega)\approx \ell (\Q).$ The construction of a Whitney decomposition satisfying the aforementioned properties follows standard arguments, wherein cuboids are employed instead of the cubes used in \cite[Lemma B.1]{T21}.

Let $H_0$ be the horizontal hyperplane through the origin and $\Pi$ denote the orthogonal projection on $H_0$. Consider $B_0$ a ball centered on $\Sigma$ such that $M_0^2B_0\cap\Omega\subset B\cap\Omega$ for some $M_0$ sufficiently large will be fixed below.

Now we will introduce a tree structure of parents, children and generations to this Whitney cuboid decomposition.  
 We choose $R_0\in\mathcal{W}$ such that $R_0\subset \frac{M_0}{2} B_0.$ It will be the root of the tree and we define $\mathcal{D}_{\mathcal{W}} ^{0}(R_0).$ 
 For $k\ge 1,$ we define first $$ J(R_0)=\left \{ \Pi(Q): \Q\in \W \text{ such that }\Pi(Q)\subset \Pi(R_0)\text{ and $\Q$ is below $R_0$}  \right \}. $$
 We have that $J(R_0)$ is a family of $d-1$ dimensional dyadic cubes in $H_0,$ all of them contained in $\Pi(R_0)$. Let $J_k(R_0)\subset J(R_0)$ be the family of $(d-1)-$dimensional dyadic cubes in $H_0$ with side length equal to $2^{-k}\ell(R_0).$ To each $\Q'\in J_k(R_0)$ we assign some $\Q\in\W$ such that $\Pi(\Q)=\Q'$ and such that $\Q$ is below $R_0$ (notice that there may be more than one choice for $\Q$ but the choice is irrelevant), see \cite[Lemma B.2]{T21}, and we write $s(\Q')=\Q.$ Then we define $$ \mathcal{D}_{\W}^{k}(R_0):=\left \{ s(\Q'):\Q'\in J_k(R_0) \right \} $$ and $$ \mathcal{D}_{\W}(R_0)=\bigcup_{k\ge0}\mathcal{D}_{\W}^{k}(R_0).$$ Notice that, for each $k\ge 0$, the family $\left \{ \Pi(\Q): \Q \in \mathcal{D}_{\W}^{k}(R_0) \right \} $ is a partition of $\Pi(R_0)$. Finally, for each $R\in \mathcal{D}_{\W}^{k}(R_0)$ and $j\ge 1,$ we denote $$ \mathcal{D}_{\W}^{j}(R)=\left \{\Q\in \mathcal{D}_{\W}^{k+j}(R_0): \Pi(Q)\subset\Pi(R)  \right \}. $$ By the properties of Whitney cuboids, we can observe that $$ \Q\in \mathcal{D}_{\W}(R_0)\Rightarrow 
 \txt{dist}(Q,\p\Omega)=\txt{dist}(\Q, B\cap\p\Omega)\approx \ell(\Q).$$ Further, for any $\Q\in\W,$ we define its associated cylinder by $$ \mathcal{C}(\Q):=\Pi^{-1}(\Pi(\Q))$$ and the $(d-1)$-dimensional Lebesgue measure on the hyperplane $H_0$ by $m_{d-1}.$ 
 \subsection{Lemma on the behavior of doubling index in the Whitney tree}
Before proving the first key lemma that probabilistically controls the behavior of the doubling index in the tree of Whitney cuboids defined in the last section, we provide some basic tools for the proof.

In establishing the monotonicity of the doubling index, we require the \(A\)-starshape condition. For locally flat \(C^1\) domains, as discussed by Gallegos \cite{G22} and Tolsa \cite{T21}, we ensure that interior points satisfy the \(A\)-starshape condition by constraining the slope to be sufficiently small. Inspired by the work of Zhu-Zhuge \cite{ZZ23}, we achieve satisfaction of the \(A\)-starshape condition for interior points by considering small regions within quasiconvex domains.

\begin{lemma}\label{shift control doubling}
Let $M_0\ge2S>0,T>0$ be fixed sufficiently large positive constants, and let $R\in \mathcal{D}_{\W}(R_0)$ satisfy
\begin{equation}\label{smallcon}
S^2\ell(R)+(\sqrt{1+L^2}T+S)\omega((\sqrt{1+L^2}T+S)\ell(R))
\le \frac{1}{\gamma\Lambda(1+L^2)T}.
\end{equation}
 Then for any $x\in \frac{M_0}{2}B_0\cap \Omega$ satisfying
 $$ dist(x,\p\Omega)\le T\ell(R)\quad \txt{and}\quad dist(x,\p\Omega)\ge T^{-1}\ell(R),$$
  we have $B(x, S\ell(R))\cap  \Omega$ is A-starshaped with respect to $x$.    
\end{lemma}
\begin{proof}
We begin by noting that since $M_0\ge 2S$ is big enough such that $B(x, S\ell(R))\cap \p\Omega\subset B\cap \p\Omega$ can be expressed as a Lipschitz graph.
    To prove the Lemma, it suffices to check that, in the situation above, $$ n(y)\cdot A(y)A^{-1}(x)(y-x)\ge 0\text{ for almost every }y\in \p\Omega\cap B(x, S\ell(R)).$$ Let $x'=\Pi(x)$ be the orthogonal projection of $x$ on $\p\Omega$. Obviously, due to the fact $\text{dist}(x,\p\Omega)\ge T^{-1}\ell(R),$ we can restate $$x=x'+\norm{x-x'}e_d\quad \text{ and }\quad \norm{x-x'}\ge T^{-1}\ell(R).$$ 
    On the other hand, consider a point $y\in\p\Omega\cap B(x,S\ell(R))$ such that the normal $n=n(y)$ exists,
    according to  the definition of quasiconvex domains in (\ref{quasiconvex}) and $\txt{dist}(x,\p\Omega)\le T\ell(R)$ we have the inequalities: 
    \begin{equation}\begin{aligned}\label{quasiconvexdef}
    &\norm{x'-y}\le \norm{x-x'}+\norm{x-y}\le \sqrt{1+L^2}T\ell(R)+S\ell(R)\\
        &n(y)\cdot (x'-y)\le (\sqrt{1+L^2}T+S)\ell(R)\omega((\sqrt{1+L^2}T+S)\ell(R)).
        \end{aligned}
    \end{equation}
   Considering the assumptions on \(A\) and the conditions specified in (\ref{quasiconvex}) and (\ref{quasiconvexdef}), we can deduce the following result:
    \begin{equation*}\begin{aligned}
       &n(y)\cdot A(y)A^{-1}(x)(y-x)\\&\ge n(y)\cdot (x'-x)+ n(y)\cdot (y-x')-\gamma\Lambda\norm{y-x}^2\\
       &\ge \frac{\ell(R)}{T\sqrt{1+L^2}}-(\sqrt{1+L^2}T+S)\ell(R)\omega((\sqrt{1+L^2}T+S)\ell(R))-\gamma\Lambda(S\ell(R))^2\\
       &\ge \ell(R)(\frac{1}{T\sqrt{1+L^2}}-(\sqrt{1+L^2}T+S)\omega((\sqrt{1+L^2}T+S)\ell(R))-\gamma\Lambda S^2\ell(R))\ge 0.
    \end{aligned} 
    \end{equation*}    
\end{proof}

Next, we recall the regularity of elliptic equation over Lipschitz domain; see \cite{K94,KS11} and references therein. Assume that $A$ is H$\ddot{o}$lder continuous and let $u$ be an $A$-harmonic function in $\Omega\cap B_2$ and $u|_{\p\Omega\cap B_2}$ in the sense of trace. Then the nontangential maximal function $(\nabla u)^{*}|_{\p\Omega\cap B_1}\in L^{2}(\p\Omega\cap B_1).$ This particularly implies that for almost every $x\in \p\Omega\cap B_1, \nabla u(y)\to \nabla u(x)$ as $\Omega\ni y\to x$ nontangentially. Therefore, $\frac{\p u}{\p\nu}=n\cdot A\nabla u$ exists almost everywhere, thus confirming the well-defined nature of the following quantitative Cauchy uniqueness over Lipschitz domains in \cite{ZZ23}. We point out that this lemma is an important tool for the proof of the Lemma \ref{keylemma1}.
\begin{lemma}\label{Cauchyuniquess}
  Let $\Omega$ be a Lipschitz domain and $0\in \p\Omega.$ There exists $0<\tau =\tau(d,L,A)<1$ such that if u is an A-harmonic function in $B_{1,+}$ satisfying $\Norm{u}_{L^2(B_{1,+})}=1$ and $\norm{u}+\norm{n\cdot A\nabla u}\le \epsilon\le 1$ on $B_1\cap \p\Omega,$ then$$ \Norm{u}_{L^{\infty}(B_{1,+})}\le C\epsilon^{\tau},$$ where $C$ depends only on $d, L$ and $A$.

\end{lemma}

We now present the first key lemma needed to prove Theorem \ref{Minskithm1}. This lemma controls the behavior of the doubling index in the tree of Whitney cuboids defined in the previous section. The approach of our proof is similar to \cite[Lemma 3.1]{T21} for harmonic functions in \(C^1\) domains and \cite[Lemma 4.1]{G22} for general elliptic equations in \(C^1\) domains. Additionally, we compare this lemma with the well-known hyperplane lemma found in  \cite[Lemma 7]{LMNN21} and \cite[Lemma 4.2]{ZZ23}. Note that the following key lemma quantifies the probability of a decrease in the doubling index as one approaches the boundary.

Define $N_u^*(x,r)=N_u(x,r)+1$.
\begin{lemma}\label{keylemma1}
    Let $N_0>1$ be big enough. There exists some absolute $\delta>0$ such that for all $S\gg 1$ big enough the following holds. Let $R$ be a cuboid in $\mathcal{D}_{\W}(R_0)$ with $\ell(R)
    $ small enough depending on $S$ and $\delta$ that satisfies $N_u^*(x_R, S\ell(R))\ge N_0.$ Then, there exists some positive integer $K=K(S)$ big enough such that if we let $$ \mathcal{G}_{K}(R)=\left \{ \Q\in\mathcal{D}_{\W}^{K}(R): N_u^*(x_Q,S\ell(Q))\le \frac{1}{2}N_u^*(x_R, S\ell(R))\right \} $$ then :
    \begin{itemize}
    \item [(1)] $m_{d-1}(\bigcup_{\Q\in \mathcal{G}_{K}(R)}\Pi(\Q))\ge \delta m_{d-1}(\Pi(R)),$
    \item [(2)] for $\Q\in\mathcal{D}_{\W}^{K}(R)$, it holds: $N_u^*(x_Q,S\ell(Q))\le (1+CS^{-1})N_u^*(x_R,S\ell(R))$.
\end{itemize}

    Note that (2) does not require $N_u^*(x_R, S\ell(R))\ge N_0.$ It is important that $\delta $ does not depend on S and $N_0$ only depends on $d$, $\omega, L$ and $A$.
\end{lemma}

\begin{proof}[Proof of Lemma \ref{keylemma1}]
   We choose $R\in \mathcal{D}_{\W}(R_0)$ with $\ell(R)$ small enough. Let $$\Gamma=\left \{x=(x',x_d): x'\in \Pi( \frac 12 R),x_d=\phi(x') \right \}. $$ For some $j\gg 1$ to be fixed below depend only on $d$, $\omega,L$ and $A$. Define $$ \Gamma_1=\Gamma + 2^{-j}\ell(R) e_d=\left \{x=(x',x_d): x'\in \Pi( \frac 12 R),x_d=\phi(x')+2^{-j}\ell (R) \right \}.$$ From now on, we will denote by $J$ the family of cuboids from $\W$ that intersect $\Gamma_1$. By our construction of the Whitney cuboids and the definition of $\Gamma_1$, choosing $j$ big enough, we have $$ \ell(Q)\approx 2^{-j}\ell (R)\quad \text{and}\quad  \Pi(Q)\subset\Pi(R) \text{ for all }Q\in J.$$ Denote by $\txt{Adm}(2WQ)$ the set of points $x\in \Omega\cap 2WQ$ and $B(x, \txt{diam}(100\Lambda WC^* Q))\cap \Omega$ is $A$-starshaped with respect to $x.$ Recall that $W$ is one of the constants in the definition of Whitney cuboids depending only on $d$, $\Lambda$ is the ellipticity constant of $A$, and $C^*$ is the constant in Lemma \ref{sshiftcenter}. We assume that  $\ell(Q)$ is small enough (depending on $d, \omega, A$) so that $3Q\subset \text{Adm}(2WQ)$(using Lemma \ref{shift control doubling}). Then by Lemma \ref{sshiftcenter}, we have 
   \begin{equation}\label{property 1}
       \sup_{x\in \txt{Adm}(2WQ)} N_{u}^*(x,\txt{diam}(10WC^*Q))\le C_0N_{u}^*(x_Q,\txt{diam}(20WC^*Q))
   \end{equation}
   where $C_0$ is an absolute constant. Note that $$\norm{x-x_Q}\le \txt{diam}(WQ)< \frac{\txt{diam}(10WC^*Q)}{C^*}$$ since $x\in 2W\Q$, hence it satisfies the condition required by Lemma \ref{sshiftcenter}. 
   \begin{claim}
       There exists some $Q\in J$ such that 
       \begin{equation}\label{claim}
       N_{u}^*(x_Q,\txt{diam}(20WC^*Q))\le \frac{N_{u}^*(x_R, S\ell(R))}{4C_0^2}
   \end{equation}
   if $j$ is big enough (but independent of $S$) and we assume that $\ell(R)$ is small enough depending on $j, S, \omega, L, d$ and $A.$
   \end{claim}
   \begin{proof}[Proof of the claim]
      From now on, we denote $N^*=\frac{N_{u}^*(x_R, S\ell(R))}{4C_0^2}$, $\ell_Q=\txt{diam}(20WC^*Q)$, $\ell_R=\txt{diam}(20WC^*R)$.
      We prove by contradiction. Assume for all $Q\in J, N_{u}^*(x_Q,\ell_Q)>N^*$. By the fact $\ell_Q\approx 2^{-j}\ell_R$, by choosing $j$ big enough, there exists $j_0=j_0(d, L)>0$ such that $2^{j-j_0}\ell_Q\le \frac12 \ell_R\le 2^{j+j_0}\ell_Q.$ We may let $$(1+C\gamma \ell_R)^{j+j_0}\le 2\quad \txt{and}\quad 2^{5j}(\omega(2^{4j}\ell_R)+2^{4j}\ell_R)\ll 1.$$ Below, we will make repeated use of this property, often without further reference. By Lemma \ref{shift control doubling} and the almost monotonicity in Lemma \ref{Keymdoub}, we have for all $1\le i\le j+j_0,$
      $$ N_{u}^*(x_Q, \ell_Q)\le (1+C\gamma \ell_R)^{i}N_u^*(x_Q, 2^i\ell_Q)\le 2N_u^*(x_Q, 2^i\ell_Q).$$
      As a result,$$ N_u(x_Q, 2^i\ell_Q)>\frac{N^*-2}{2}.$$
      Thus, we obtain that
      \begin{equation*}
          \begin{aligned}
              J_u(x_Q, 4\ell_Q)&=J_u(x_Q, 2^{j-j_0}\ell_Q)\txt{exp}(-\sum_{i=2}^{j-j_0-1}N_{u}(x_Q,2^{i}\ell_Q))\\
              &\le J_u(x_Q,\frac{1}{2}\ell_R)\txt{exp}(-(j-j_0-2)\frac{N^*-2}{2})\\
              & \le J_u(\Tilde{x}_R,\ell_R)\txt{exp}(-(j-j_0-2)\frac{N^*-2}{2}).
          \end{aligned}
      \end{equation*}
In view of the definition of $J_u$ in (\ref{defJ1}) and (\ref{defJ2}), we have 
\begin{equation}\label{cauchy1}
    \int_{B(x_Q, \frac{4}{\Lambda}\ell_Q)}u^2\le CJ_u(\Tilde{x}_R,\ell_R)\txt{exp}(-\frac{jN^*}{8}),\quad \text{for each }Q\in J,
\end{equation}
where we have chosen $j\ge 2j_0+4,$  $N_0\ge 16C_0^2$ such that $(j-j_0-2)\frac{N^*-2}{2}\ge \frac{jN^*}{8}$.  Then (\ref{cauchy1}) implies 
\begin{equation}\label{cauchy1*}
    \sup_{y\in\Gamma_1}\sup_{B(y,\frac{1}{\Lambda}\ell_Q)} \norm{u}\le CJ_u^{\frac 12}(\Tilde{x}_R,\ell_R)(\ell_Q)^{-d/2}e^{-\frac{jN^*}{16}}.
\end{equation}
Here, we utilize \(C^* \geq \Lambda\) to ensure that \(B(\Gamma_1, \frac{1}{\Lambda}\ell_Q) \subset \bigcup_{Q \in J} B(x_Q, \frac{2}{\Lambda}\ell_Q)\).
By the interior estimate and (\ref{cauchy1*}), we have 
\begin{equation}\label{cauchy2}
    \sup_{\Gamma_1}\norm{u}+\ell(Q)\sup_{\Gamma_1}\norm{\nabla u}\le C(\ell(Q))^{-\frac{d}{2}}J_u^{\frac 12}(\Tilde{x}_R,\ell_R)e^{-\frac{jN^*}{16}}.
\end{equation}
Denote $\Tilde{x}_R$ as the projection of $x_R$ on $\Gamma_1$.
Let $$\Tilde{B}_{\frac{\ell(R)}{2},+}(\Tilde{x}_R)=\left \{ x=(x',x_d):\norm{x-\Tilde{x}_R}\le\frac{1}{2}\ell(R), x_d>\phi(x')+2^{-j}\ell(R)e_d \right \}.$$
Therefore, applying Lemma \ref{Cauchyuniquess} in $\Tilde{B}_{\frac{\ell(R)}{2},+}(\Tilde{x}_R)$ with Lipschitz boundary $\Gamma_1$, we obtain from (\ref{cauchy2}) that 
\begin{equation}\label{cauchy3}
    \sup_{\Tilde{B}_{\frac{\ell(R)}{4},+}(\Tilde{x}_R)}\norm{u}\le C(\ell(R))^{-\frac{d}{2}}2^{\tau jd/2+\tau j}J_u^{\frac 12}(\Tilde{x}_R,\ell_R)e^{-\frac{\tau jN^*}{16}},
\end{equation}
where we also used the fact $\ell(Q)\approx 2^{-j}\ell(R)$. Combine with (\ref{cauchy1*}) again, (\ref{cauchy3}) yields 
\begin{equation}\label{cauchy4}
    J_{u}(\Tilde{x}_R, \frac{\ell(R)}{8\Lambda})\le  C2^{\tau jd+jd}J_u(\Tilde{x}_R,\ell_R)e^{-\frac{\tau jN^*}{8}}.
\end{equation}
Finally, we relate $ J_{u}(\Tilde{x}_R, \frac{\ell(R)}{8\Lambda})$ to $J_u(\Tilde{x}_R,\ell_R)$ by a sequence of doubling inequalities. Let $m$ be the smallest integer such that $2^m\frac{\ell(R)}{8\Lambda}>\ell_R.$ Note that $m$ only depends on $d, A,$ and $L.$ By the almost monotonicity of the doubling index, for $1\le i\le m,$
$$ N_u^*(\Tilde{x}_R,2^{-i}\ell_R)\le 2N_u^*(\Tilde{x}_R,\frac{1}{2}\ell_R.$$ 

On the other hand, let S be large enough and $\log_2 S=[\log_2 S]$, by choosing $\ell_R$ small enough depends on $j, S, d, L$ and $A$, by Lemma \ref{Keymdoub} and Lemma \ref{sshiftcenter}, we can obtain that 
$$N_u^*(\Tilde{x}_R,\frac{1}{2}\ell_R\le C_0N_u^*(x_R,\ell_R)\le 2C_0N_u^*(x_R, S\ell(R))=8C_0^3N^*.$$
This yields
\begin{equation}\label{cauchy5}
    \begin{aligned}
        \log\frac{J_u(\Tilde{x}_R,l_R)}{J_u(\Tilde{x}_R,\frac{\ell(R)}{8\Lambda})}&\le \log \frac{J_u(\Tilde{x}_R,\ell_R)}{J_u(\Tilde{x}_R,2^{-m}\ell_R)}\\
        &\le \sum_{i=1}^{m} N_u(\Tilde{x}_R, 2^{-i}\ell_R)\le 2m N_u^{*}(\Tilde{x}_R, \frac{1}{2}\ell_R)\\
        &\le 16mC_0^3N^*.
    \end{aligned}
\end{equation}
Comparing (\ref{cauchy4}) and (\ref{cauchy5}), we obtain $$ (\frac{\tau}{8}-\frac{16mC_0^3}{j})N^*\le \frac{C}{j}+ (\tau d+d)\log2.$$
Clearly, this is a contradiction if $j$ and  $N_0$ are big enough. Note that $j$ and $N_0$ depends only on $d, L, $ and $A$.
   \end{proof}
Now we are ready to introduce the set $\hat{\mathcal{G}}_K(R).$ Fix $Q_0\in J$ such that (\ref{claim}) holds for $Q_0$. Define $$ \hat{\mathcal{G}}_K(R)= \left \{Q\in \mathcal{D}_{\W}^{j+k}(R):\Pi(Q)\subset\Pi(Q_0)  \right \}\txt{ with }k=[\log_2 S].$$
So we have $\hat{\mathcal{G}}_K(R)\subset \mathcal{D}_{\W}^{K}(R)$ with $K=k+j$ and it holds $S\ell(Q)\approx \ell(Q_0)$ for every $Q\in \hat{\mathcal{G}}_K(R)$. If $P\in \hat{\mathcal{G}}_K(R),$ taking into account that $x_P\in\txt{Adm}(2WQ_0)$ for $S$ sufficiently big. Hence by (\ref{property 1}) and (\ref{claim}), we have 
\begin{equation*}
    \begin{aligned}
        N_u^*(x_P, S\ell(P))&\le C_0N_u^*(x_P, \txt{diam}(10WC^*Q_0))\\
        &\le C_0^2N_u^*(x_{Q_0}, \txt{diam}(20WC^*Q_0))\le \frac{N_{u}^*(x_R, S\ell(R))}{4} \\
        & \le \frac{1}{2}N_{u}^*(x_R, S\ell(R)).
    \end{aligned}
\end{equation*}
Notice also that $$m_{d-1}(\bigcup_{\Q\in \mathcal{G}_{K}(R)}\Pi(\Q))\ge m_{d-1}(\bigcup_{\Q\in \hat{\mathcal{G}}_K(R)}\Pi(\Q))= (\ell(Q_0))^{d-1}\approx (2^{-j}\ell (R))^{d-1}$$
and recall that $j$ is independent of $S$. So (1) holds with $\delta\approx 2^{-j(d-1)}$.

The property (2) is a consequence of Lemma \ref{Keymdoub} and Lemma \ref{sshiftcenter}. Indeed, for any $P\in \mathcal{D}_{\W}^K(R),$ since $\norm{x_P-x_R}\preceq \ell(R)$, taking $\frac{\theta}{S\ell(R)}\approx \frac{1}{S}$ in the Lemma \ref{sshiftcenter}, we deduce that 
\begin{equation*}
    \begin{aligned}
        N_u^*(x_P, S\ell(P))&\le (1+ C\gamma S\ell(R))^K N_u^*(x_P, \frac{S}{2} \ell(R))\\
        &\le (1+ C\gamma S\ell(R))^{K} (1+C\gamma S\ell(R)+CS^{-1})N_{u}^*(x_R, S\ell(R))\\
        &\le (1+CS^{-1})N_{u}^*(x_R, S\ell(R)),
    \end{aligned}
\end{equation*}
  provided that $\ell(R)$ sufficiently small depends on $j, S, \omega, L, d$ and $A.$
\end{proof}
\section{ABSENCE OF NODAL POINTS}\label{sectionnozeros}
In this section, we will establish the second main lemma concerning the behaviour of $N$ near the boundary. This lemma shows that if we have a ball near the boundary with bounded doubling index, then we can find a smaller ball centered at the boundary where $u$ does not change sign. Note that, in what follows, we refer to the doubling index $N_u(x,r)$ in section \ref{sectionmonotonicityofdoublingindex}. Moreover, we consider the Lipschitz quasiconvex domain $\Omega$ with the Whitney structure defined in section \ref{Whitneydecomposition}.

First, we will recall two useful lemmas in \cite{ZZ23}.
\begin{lemma}\label{boundarydoubling}
  Let $x_0\in \p\Omega.$ Then there exist $C=C(\Lambda, L)>0$ and $r_0=r_0(\Lambda,L,\gamma,\omega)>0$ such that for $r<r_0$
$$ N_u(x_0, r)\le (1+C(\gamma r+\omega(16r)))N_u(x_0, 2r)+C(\gamma r+\omega(16 r)).$$
\end{lemma}
The following boundary lemmas that show the absence of zeros near relatively large portion of the boundary for small doubling index. We can find the detailed proof in \cite{ZZ23}.
\begin{lemma}\label{citeabsence}
    Let $\Omega$ be a quasiconvex Lipschitz domain and $0\in \partial\Omega.$ Moreover, $\p\Omega\cap B(0, 32\Lambda)$ can be expressed as a Lipschitz graph with constant $L.$ Let $u$ be A-harmonic in $B(0,16\Lambda)\cap \Omega$ and $u=0$ on $B(0,16\Lambda)\cap \p\Omega.$ Suppose
    $$ N_u(0,\frac{1}{2})\le N.$$
    Then there exist $\omega_0>0, \gamma_0>0, \rho>0$, depending only on $N, L, \Lambda$ and $d$, such that if $\omega(32\Lambda)\le \omega_0$ and $\norm{\nabla A}\le \gamma_0,$ then $\norm{u}>0$ in $B(y,\rho)\cap \Omega$ for some $y\in \p\Omega\cap B_{1/8}.$
\end{lemma}
Now, we are ready to give the second key lemma.
\begin{lemma}\label{copykey}
Let $\Omega$ be a quasiconvex Lipschitz domain and fix some $x_0\in\p\Omega.$ 
    For any $N>0$ and $S\gg 1$ large enough, there exist positive $r_0=r_0(N, S, d, L,\Lambda,\omega,\gamma)$ and $\rho=\rho(N, d, L, \Lambda,\omega,\gamma)$ such that the following statement holds. Let $u$ be A-harmonic in $\Omega$ and $u=0$ on $B(x_0,2r_0)\cap\p\Omega$. Consider $R\in \W$ such that $R\subset B(x_0,\frac{3}{2}r_0)\cap \Omega$. If $Q$ is a cuboid in $\mathcal{D}_{\W}(R)$ satisfying $N(x_Q, S\ell(Q))\le N,$ then there exists a ball $\Tilde{B}$ centered in $\p\Omega \cap \mathcal{C}(Q)$ with radius $\rho\ell(Q)$ such that $u$ does not vanish in $\Tilde{B}\cap\Omega.$
\end{lemma}
\begin{proof}[Proof of Lemma \ref{copykey}]
  When we choose $\ell(R)$ small enough depending on $S,\omega,\gamma, \Lambda,d,$ and $L$, by Lemma \ref{shift control doubling} and Lemma \ref{sshiftcenter}, we can obtain 
  \begin{equation}\label{xqshift1}
      N_u^*(x_Q, \txt{diam}(C^*WQ))\le 2N_u^*(x_Q, S\ell(Q))
  \end{equation}

  On the other hand, we denote $\hat{x}_Q$ as the vertical projection of $x_Q$ on $\p\Omega$, by Lemma \ref{boundarydoubling}, we have that
  \begin{equation}\label{xqshift2}
    N_u^*(\hat{x}_Q, \frac{1}{2}\ell(Q))\le C N_u^*(\hat{x}_Q, \frac{1}{2}\txt{diam}(C^*WQ)) 
  \end{equation}
By Lemma \ref{sshiftcenter}, combining (\ref{xqshift1}) and (\ref{xqshift2}), we have 
\begin{equation}
    \begin{aligned}
        N_u(\hat{x}_Q, \frac{1}{2}\ell(Q))&\le C  N_u^*(\hat{x}_Q, \frac{1}{2}\txt{diam}(C^*WQ)) \\
        &\le C N_u^*(x_Q, \txt{diam}(C^*WQ))\\
        &\le CN_u^*(x_Q, S\ell(Q))\\
        &\le C(N+1).
    \end{aligned}
\end{equation}
Now, we assume that $\ell(Q)$ by rescaling the domain and the Whitney cuboids structure. Hence, we take $\hat{x}_Q$ in place of $0$ and use Lemma \ref{citeabsence}, we can find $\rho$ depending on $N, L, d,\omega,A$.
\end{proof}
To end this section, we restate the above lemma in a language closer to Lemma \ref{keylemma1}.
\begin{lemma}\label{keylemma2}
   Let $\Omega$ be a quasiconvex Lipschitz domain and fix some $x_0\in\p\Omega.$ 
    For any $N>0$ and $S\gg 1$ large enough, there exist positive $r_0=r_0(N, S, d, L,\Lambda,\omega,\gamma)$ and $K=K(N, d, L, \Lambda,\omega,\gamma)$ such that the following statement holds. Let $u$ be A-harmonic in $\Omega$ and $u=0$ on $B(x_0,2r_0)\cap\p\Omega$. Consider $R\in \W$ such that $R\subset B(x_0,\frac{3}{2}r_0)\cap \Omega$. Let $Q$ be a cuboid in $\mathcal{D}_{\W}(R)$ satisfying $N(x_Q, S\ell(Q))\le N,$  Then, for any $\Tilde{K}\ge K,$ there exist cuboids $Q''_{1}, \dots, Q''_{2^{(d-1)(\Tilde{K}-K)}}$ such that, for all $j$,
\begin{itemize}
    \item [(1)] the center of $Q''_{j}$ lies in $\p\Omega\cap B(x_0,2r_0)$,
    \item [(2)] $u|_{Q''_{j}\cap \Omega}$ does not have zeros,
    \item[(3)] there exists $Q'_{j}\in \mathcal{D}_{\W}^{\Tilde{K}}(Q)$ such that $Q''_j$ is a vertical translation of $Q'_j.$
\end{itemize}
   In particular, there exists $\rho_0=\rho_0(N, d, L, A,\omega)>0$ such that $$m_{d-1}(\bigcup_{i}\Pi(Q'_i))\ge \rho_{0}m_{d-1}(\Pi(Q)) .$$
\end{lemma}

These cuboids $Q''_j$ are the cuboids that are contained in the ball $\Tilde{B}$ given by Lemma \ref{copykey} and are vertical translation of cuboids in $ \mathcal{D}_{\W}^{\Tilde{K}}(Q)$. From now on, given a cuboid $Q\in\mathcal{D}_{\W}(R)$, we denote by $t(Q)$ the unique cuboid $Q'$ such that its center lies on $\p\Omega\cap B(x_0,2r_0)$ and $Q'$ is a vertical translation of $Q.$

\section{PROOF OF THEOREM 1.3\label{sectionoftheorem}}
Let $R$ be a cuboid in $\W$ such that it satisfies the condition of Lemma \ref{keylemma1} and Lemma \ref{keylemma2} with $S_1$ large enough (we will specify it later).
The use of Lemma \ref{keylemma1} give us constants $K_1=K_1(S_1), \delta_1:=\delta$ (that does not depend on $S_1$). Consider Lemma \ref{keylemma2} with  fixed $N=N_0$ (where $N_0$ is the constant of Lemma \ref{keylemma1}) and $S_2$ . This lemma gives us constants $K_2$ and $\delta_2:=\rho_0$ (both does not depend on $S_2$). In particular, we may assume $S_1$ larger enough than $S_2$ such that $K_2$ is smaller or equal than $K_1$. For the rest of this section, we will denote $\epsilon:=CS_1^{-1}, K=K_1(S_1)$ and $\delta_0=\txt{min}(\delta_1,\delta_2).$ We point out that the combination techniques employed below are similar to \cite[Section 6]{G22}. 

Define 
$$ N(Q):=N_u^*(x_Q, S_1\ell(Q))$$ for any $Q\in \mathcal{D}_{\W}^{jK}(R)$ for $j\ge 0.$ We define the modified doubling index $N'(Q)$ for $Q\in \mathcal{D}_{\W}^{jK}(R)$ for $j\ge 0,$ inductively. For $i=0$, we define $N'(R)=\txt{max}(N(R),\frac{N_0}{2}).$ Assume we have $N'(P)$ defined for all cubes $P\in \mathcal{D}_{\W}^{iK}(R)$ for $0\le i<j.$ Fix $\hat{Q}\in \mathcal{D}_{\W}^{(j-1)K}(R)$ and consider its vertical translation $t(\hat{Q})$ centered on $B\cap \p\Omega.$ Then:
\begin{itemize}
    \item[(a)]  if $u$ restricted to $t(\hat{Q})\cap\Omega$ has no zeros, define $N'(Q)=\frac{N'(\hat{Q})}{2}$ for $[\delta_0\txt{card}\left \{ \mathcal{D}_{\W}^K(\hat{Q})\right \} ]$ of the cubes $Q$ in $\mathcal{D}_{\W}^K(\hat{Q})$ and $N'(Q)=(1+\epsilon)N'(\hat{Q})$ for the rest of cubes in $\mathcal{D}_{\W}^K(\hat{Q})$ (the particular choice is irrelevant),
    \item[(b)] if $u$ restricted to $t(\hat{Q})\cap\Omega$ has zeros, choose $Q\in \mathcal{D}_{\W}^K(\hat{Q}),$ and 
    \begin{itemize}
        \item [1.] if its vertical translation $t(Q)$ satisfies that $u$ restricted to $t(Q)\cap\Omega$ has no zeros, define $N'(Q)=\frac{N'(\hat{Q})}{2}$,
        \item[2.] else define $N'(Q)=\txt{max}(N(Q),N_0/2).$
    \end{itemize}
\end{itemize}
Note that if a cuboid $Q$ satisfies that $u$ restricted to $t(Q)\cap\Omega$ has no zeros, then all its descendants in the Whitney cuboid structure will satisfy the same property and (a) applies to them. We point out that if $Q\in \mathcal{D}_{\W}^K(\hat{Q})$, then by the definition of cuboids, we can obtain that $t(Q)\subset t(\hat{Q}).$ On the other hand, if a cube $Q$ satisfies that $u$ restricted to $t(Q)\cap\Omega$ changes sign, then all its predecessors in the Whitney cuboid structure will satisfy the same and (b2) applies to them.

Now, a combination of Lemma \ref{keylemma1} and Lemma \ref{keylemma2} yields the following behavior for $N'(Q)$ for $Q\in\mathcal{D}_{\W}^{jK}(R), j\ge 0.$ Consider a cuboid $\hat{Q}$ and its vertical translation $t(\hat{Q}).$ Then
\begin{itemize}[label=\textbf{\textbullet}]
    \item If $u$ restricted to $t(\hat{Q})\cap\Omega$ has zeros and $N(\hat{Q})\ge N_0,$ then Lemma \ref{keylemma1} tells us at least $[\delta_0\text{card}\left \{ \mathcal{D}_{\W}^{K}(\hat{Q}) \right \} ]$ cuboids in $\mathcal{D}_{\W}^{K}(\hat{Q})$ satisfy $N(Q)\le N(\hat{Q})/2.$ Moreover, in this case, (b1) implies $N'(Q)=\frac{N'(\hat{Q})}{2}$. (b2) implies $N'(Q)=\text{max}(N(Q),N_0/2)\le N(\hat{Q})/2=N'(\hat{Q})/2$ where we have used that $N(\hat{Q})\ge N_0$. For the rest cuboids $Q$ in $\mathcal{D}_{\W}^{K}(\hat{Q}),$ (b1) implies $N'(Q)=\frac{N'(\hat{Q})}{2}\le (1+\epsilon)N'(\hat{Q})$,  (b2) implies $N'(Q)=\txt{max}(N(Q),N_0/2)\le (1+\epsilon)N(\hat{Q})=(1+\epsilon)N'(\hat{Q})$ where we have used again that $N(\hat{Q})>N_0$.
    \item If $u$ restricted to $t(\hat{Q})\cap\Omega$ has zeros and $N(\hat{Q})< N_0,$ then Lemma \ref{keylemma2} enters in play and it tells us that at least $[\delta_0\text{card}\left \{ \mathcal{D}_{\W}^{K}(\hat{Q}) \right \} ]$ cuboids $Q$ in $\mathcal{D}_{\W}^{K}(\hat{Q})$ satisfy that $t(Q)\cap\Omega$ does not contain zeros of $u.$ For these cuboids, $N'(Q)=\frac{1}{2}N'(\hat{Q})=\frac{1}{2}\txt{max}(N(\hat{Q}),N_0/2)<N_0/2.$ For the rest of cuboids $Q$ in $\mathcal{D}_{\W}^{K}(\hat{Q})$, we have $N'(Q)=\txt{max}(N(Q),N_0/2)\le\txt{max}((1+\epsilon)N(\hat{Q}),N_0/2)\le (1+\epsilon)N'(\hat{Q})$.
    \item If $u$ restricted to $t(\hat{Q})\cap\Omega$ has no zeros, then we have defined $N'(Q)=N'(\hat{Q})/2$ for $[\delta_0\text{card}\left \{ \mathcal{D}_{\W}^{K}(\hat{Q}) \right \} ]$ cuboids in $\mathcal{D}_{\W}^{K}(\hat{Q})$. We have defined $N'(Q)=(1+\epsilon)N'(\hat{Q})$ for the rest of cuboids $Q$ in $\mathcal{D}_{\W}^{K}(\hat{Q})$.
\end{itemize}
Summing up, for $\hat{Q}\in \mathcal{D}_{\W}(R)$ and random $Q\in \mathcal{D}_{\W}^{K}(\hat{Q})$, we have 
\begin{equation*}
N'(Q) \le
\begin{cases}
\begin{aligned}
& N'(\hat{Q})/2, \quad \text{with probability at least }\delta_0. \\
& N'(\hat{Q})(1+\epsilon), \quad \text{with probability at most }1-\delta_0.
\end{aligned}
\end{cases}
\end{equation*}
 This is similar to the behavior of the doubling index $N$ given by Lemma \ref{keylemma1} but without the restriction $N(\hat{Q})\ge N_0.$
\begin{proof}[Proof of Theorem \ref{Minskithm1}]
First, let us explain how $r_0$ is determined.
Let $\alpha(\delta_0)>0$ be such that 
\begin{equation}\label{alphasmall}
     \frac{\delta_0}{1-\delta_0}\frac{1-\alpha}{\alpha}=3
\end{equation}
(in particular, $\alpha<\delta_0$) and $\epsilon_0(\alpha)$ such that $$ \alpha=\frac{\log(1+\epsilon_0)}{\log(1+\epsilon_0)+\log2}.$$
Now, we pick $\epsilon<\epsilon_0$ (equivalently $S_1$ large enough, recall the discussion in section 3). Observe that by choosing $\epsilon$ we also fix $K$ and $M_0.$ 
In this way, we can find \(r_0\) such that all the conditions stated in Lemma \ref{keylemma1} and Lemma \ref{keylemma2} are satisfied. 

Recall the definition in section \ref{Whitneydecomposition}, we will prove the result for the projection of a single cube $\Pi(R)$. Afterwards, we can cover any compact in $\Sigma$ by a finite union of such cuboids which leaves stable the Minkowski dimension estimate. (In fact, we can cover $\Sigma$ by finite union of $B_0$ defined in section \ref{Whitneydecomposition} and then $\Pi(B_0\cap\p\Omega)$ can be coverd by finite union of $\Pi(R)$).

For $x\in \Pi(R),$ we denote by $Q_j(x)$ the unique cuboid $Q_j\in \mathcal{D}_{\W}^{jK}(R)$ such that $x\in\Pi(Q_j)$. we will say $Q_j(x)$ is good cuboid if $N'(Q_j\le \frac{1}{2}N'(Q_j-1)$ and that is bad otherwise.
\begin{remark}
    Note that, with the previous definitions, $N'(Q)<\frac{N_0}{2}$ implies that for all $x\in t(Q)\cap\Sigma$ there is a neighborhood where $u$ does not vanish in. Thus we only need to study the Minkowski dimension of the set of points $x\in \Pi(R)$ such that they are not in $\Pi(Q)$ for some $Q\in\mathcal{D}_{\W}(R)$ with $N'(Q)<\frac{1}{2}N_0.$ Also, notice that the map $\Pi :\Sigma\cap\Pi^{-1}(\Pi(R))\mapsto \Pi(R) $  is bilipschitz and thus it preserves Minkowski dimensions.
\end{remark}

We define the goodness frequency of a point $x\in \Pi(R)$ as $$ F_j(x)=\frac{1}{j}\#\left \{\txt{good cubes in }Q_1(x), \dots, Q_j(x)  \right \} \quad \txt{for }j\in\mathbb{N}.$$
We define $\Bar{F}(x)=\limsup_{n\to\infty}F_n(x)$. 
 For all $j>1,$ define $$ \mu_j=\frac{1}{j}\log_2(\frac{2N'(R)}{N_0})\ge 0.$$
 Notice that $\mu_j\to 0$ as $j\to\infty.$
 \begin{claim}
      For $j>1,$ the following holds $$ F_j(x)\ge \alpha+\mu_j\Longrightarrow N'(Q_j(x))<\frac{N_0}{2}.$$
   \end{claim}
\begin{proof}[Proof of claim]
Note that for any $0<\epsilon<\epsilon_0,$ we have 
\begin{equation}\label{choosecondition}
   (\frac{1}{2})^\alpha(1+\epsilon)^{1-\alpha}<1. 
\end{equation}
We have that $$ N'(Q_j(x))\le (\frac{1}{2})^{j(\alpha+\mu_j)}(1+\epsilon)^{j(1-\alpha-\mu_j)}N'(R)<(\frac{1}{2})^{j\mu_j}(1+\epsilon)^{-j\mu_j}N'(R)<\frac{N_0}{2}$$
by (\ref{choosecondition}) and the definition of $F_j(x)$ and $\mu_j.$   
\end{proof}

Now, we can reduce the problem to studying the Minkowski dimension of the set of points $$ E=\left \{x\in \Pi(R)|F_j(x)<\alpha+\mu_j,\txt{ for all } j\in\mathbb{N}. \right \}. $$

If we consider a random sequence of cuboids $(Q_j)_j$ with $Q_0=R$ and $Q_j\in\mathcal{D}_{\W}^{K}(Q_{j-1}),$ and let $x\in \bigcap_{j\ge 0}\Pi(Q_j),$ the probability that $F_j(x)\le \beta_j$ for $j\in\mathbb{N}$ and $\beta_j=\alpha+\mu_j\in (0,1)$ is bounded above by $$ \sum_{i=0}^{[j\beta_j]} \binom{j}{i}\delta _0^i(1-\delta _0)^{j-i}.$$
It is important to note that the probability is uniform for any \(x\) in \(\Pi(R)\).
In what follows, by \ref{alphasmall}, we will assume that $\beta_j$ satisfy $2<\frac{\delta_0}{1-\delta_0}\frac{1-\beta_j}{\beta_j}<4$ for $j$ big enough, in particular 
$\beta_j<\delta_0$ for large $j$. Let us find an upper bound for the previous quantity for very large $j$:
$$ A_j:=\sum_{i=0}^{[j\beta_j]} \binom{j}{i}\delta _0^i(1-\delta _0)^{j-i}=(1-\delta_0)^j\sum_{i=0}^{[j\beta_j]} \binom{j}{i}(\frac{\delta_0}{1-\delta_0})^{j-i}.$$
Observe that 
$$ \binom{j}{k-1}<\frac{\beta_j}{1-\beta_j}\binom{j}{k},\quad \txt{ for }0<k\le [j\beta_j].$$
This because $$ \frac{\binom{j}{k-1}}{\binom{j}{k}}=\frac{k!(j-k)!}{(k-1)!(j-k+1)!}=\frac{k}{j-k+1}\le \frac{[j\beta_j]}{j-[j\beta_j]+1}<\frac{\beta_j}{1-\beta_j}.$$
Iterating this inequality, we obtain $$\binom{j}{i}<(\frac{\beta_j}{1-\beta_j})^{[j\beta_j]-i}\binom{j}{[j\beta_j]},\quad \txt{for }i<[j\beta_j].$$
Using  this observation we can bound $A_j$ by 
$$ (1-\delta_0)^{j}\sum_{i=0}^{[j\beta_j]}\binom{j}{i}(\frac{\delta_0}{1-\delta_0})^i\le (1-\delta_0)^j(\frac{\beta_j}{1-\beta_j})^{[j\beta_j]}\binom{j}{[j\beta_j]}\sum_{i=0}^{[j\beta_j]}(\frac{\delta_0}{1-\delta_0}\frac{1-\beta_j}{\beta_j})^i. $$
We use Stirling's formula to approximate for $j$ big enough:
\begin{equation*}
    \begin{aligned}
        \binom{j}{[j\beta_j]}&\approx \frac{\sqrt{2\pi j}(\frac{j}{e})^j}{\sqrt{2\pi [j\beta_j]}(\frac{[j\beta_j]}{e})^{[j\beta_j]}\sqrt{2\pi (j-[j\beta_j])}(\frac{j-[j\beta_j]}{e})^{j-[j\beta_j]}}\\
        &\preceq \frac{\sqrt{2\pi j}(\frac{j}{e})^j}{\sqrt{2\pi j\beta_j}(\frac{j\beta_j}{e})^{j\beta_j}\sqrt{2\pi j(1-j\beta_j)}(\frac{j(1-\beta_j)}{e})^{j(1-\beta_j)}}\\
        &\preceq \frac{1}{\sqrt{2\pi j\beta_j(1-\beta_j)}}(\frac{1}{\beta_j^{\beta_j}(1-\beta_j)^{1-\beta_j}})^j.
    \end{aligned}
\end{equation*}
We also estimate $$ (\frac{\beta_j}{1-\beta_j})^{[j\beta_j]}\preceq (\frac{\beta_j}{1-\beta_j})^{j\beta_j},$$ and $$\sum_{i=0}^{[j\beta_j]}(\frac{\delta_0}{1-\delta_0}\frac{1-\beta_j}{\beta_j})^i\approx (\frac{\delta_0}{1-\delta_0}\frac{1-\beta_j}{\beta_j})^{j\beta_j}. $$

Summing everything up, we obtain
$$ A_j\preceq \frac{2}{\sqrt{2\pi j\beta_j(1-\beta_j)}}(\frac{\delta _0^{\beta_j}(1-\delta_0)^{1-\beta_j}}{\beta_j^{\beta_j}(1-\beta_j)^{1-\beta_j}})^j,$$
for $j$ large enough. We define $z(\beta_j)=\frac{\delta _0^{\beta_j}(1-\delta_0)^{1-\beta_j}}{\beta_j^{\beta_j}(1-\beta_j)^{1-\beta_j}}.$ Observe that $z(\beta_j)<1$ for $\beta_j<\delta_0.$

Now, we will choose a suitable covering of the set $E$ by projection of cuboids in $\mathcal{D}_{\W}^{jK}(R).$ 
For $j\ge 1,$ set 
$$ E_j:=\left \{ x\in\Pi(R)| F_j(x)\le \alpha+\mu_j \right \}, $$ so that $E=\bigcap_j E_j.$ Let's upper bound the Minkowski dimension of $E$ by finding a certain cover $E_j$ by projections of cuboids in $\mathcal{D}_{\W}^{jK}(R)$ (note that there are $M=2^{(d-1)K}$ cuboids in $\mathcal{D}_{\W}^{K}(R)).$ For any $Q'\in \Pi(\mathcal{D}_{\W}^{jK}(R))$, we point out that either $Q'\cap E_j=\emptyset$ or $Q'\cap E_j=Q'.$

By employing the earlier asymptotic results, according to (\ref{alphasmall}), we can conclude that \(\beta_j < \delta_0\) for sufficiently large values of \(j\). Hence, we can cover $E_j$ (for $j$ large enough) with 
$$ CM^j\frac{2}{\sqrt{2\pi j(\alpha+\mu_j)(1-\alpha-\mu_j)}}z(\alpha+\mu_j)^j$$ projections of cuboids in $\mathcal{D}_{\W}^{jK}(R)$ and each of those cuboids has side length $M^{-j(d-1)}.$

Now we are ready to upper bound the Minkowski dimension of the set $E.$ We will use the following definition of upper Minkowski dimension 
$$ \text{dim}_{\overline{\mathcal{M}}}E=\limsup_{j\to \infty } \frac{\log\left(\#\left\{\text{$M$-dyadic cuboids $Q$ of side length $M^{-j}$ that satisfy $Q \cap E \ne \emptyset$}\right\}\right)}{j\log M} .$$
which is equivalent to the dyadic one in (\ref{defMinkovski}). By covering a single set $E_j\supset E$ and making $j\to\infty,$ we obtain the following upper Minkowski dimension estimate 
\begin{equation*}
    \begin{aligned}
      \text{dim}_{\overline{\mathcal{M}}}E&\le  \lim_{j\to\infty} \frac{j(\log M +\log z(\alpha+\mu_j))+\log (\frac{2}{\sqrt{2j\pi(\alpha+\mu_j)(1-\alpha-\mu_j)}})}{(d-1)^{-1}j\log M}\\
      &=(d-1)\frac{\log M+\log z(\alpha)}{\log M}<d-1
    \end{aligned}
\end{equation*}
 since $z(\alpha)<1.$
\end{proof}

We are now ready to present the proof of Corollary \ref{quasiconvexLinconjecture}. Theorem \ref{Minskithm1} indicates that we can decompose \(\Sigma\) into its intersection with a countable family of balls \((B_i)_i\) and a set of Hausdorff dimension smaller than \(d-1\) by exhaustively covering \(\Sigma\) with compacts. Utilizing countable additivity, our objective simplifies to proving that 
\[ \mathcal{H}^{d-1}\left(\left \{ x\in\Sigma \,|\, \partial_\nu u(x)=0 \right \} \cap B \right) = 0 \]
for any ball \(B \in (B_i)_i\) as defined in the decomposition from Theorem \ref{Minskithm1}.
Prior to commencing the proof, we introduce the concept of a $\mathcal{A}_\infty$ weight as defined in \cite{G22}.
\begin{definition}
    We say that  a measure $\mu\in\mathcal{A}_\infty(\sigma)$ if there exist $0<\lambda_1,\lambda_2<1$ such that for all balls $B$ and subset $E\subset B,\sigma(E)\le \lambda_1\sigma(B)$ implies $\mu(E)\le \lambda_2\mu(B).$
\end{definition}
\begin{proof}[Proof of Corollary \ref{quasiconvexLinconjecture}]
   Let \(B\) be centered on \(\Sigma\) such that \(u|_{B\cap\Omega}\) maintains a consistent sign. Without loss of generality, we assume \(u\) to be positive. 
   
   According to Dahlberg's theorem \cite{Da77}, the harmonic measure for the domain \(B \cap \Omega\) serves as an \(\mathcal{A}_\infty\) weight concerning surface measure. As stated in \cite{FKP91}, given the matrix's uniform ellipticity and Lipschitz coefficients, its corresponding elliptic measure \(\mu_A\) also qualifies as another \(\mathcal{A}_\infty\) weight. Notably, this implies that the density \(\frac{d\mu_A}{d\sigma}\) can only vanish within a set of surface measure zero.

   On the other hand, the density of elliptic measure is comparable with $(A\nabla G,\nu)$ at the boundary (where $G$ is the Green function with pole outside $2B$). We refer to \cite{K94} for the definition and properties of elliptic measures. Since $u$ is positive in $B\cap\Omega$, by the comparison principle for positive solutions,  we have that $A\nabla u$ on $\Sigma\cap B$ is comparable to $A\nabla G.$ This conclusion concludes the proof.
\end{proof}
\begin{remark}
    An alternative method for establishing Corollary \ref{quasiconvexLinconjecture} can be found in \cite{T21}. The ingredients are Lemma \ref{keylemma1}, and \cite[Lemma 0.2]{AE97}.  Notably, the techniques outlined in Section \ref{sectionnozeros} are not imperative for deriving this result. 
\end{remark}

\bibliographystyle{abbrv}
\bibliography{main}
\end{document}